\documentclass{amsart}

\usepackage{fancyheadings}
\usepackage{enumerate}
\usepackage{amsthm}

\usepackage{hyperref}

\usepackage{pgf,tikz}
\usetikzlibrary{arrows}

\usepackage{color}

\usepackage{graphicx} 
\usepackage{amssymb, amsthm, amsmath}

\usepackage{ mathrsfs }
\usepackage{verbatim}

\usepackage{enumerate}

\parskip=\smallskipamount

\newcommand{\R}{\mathbb{R}}
\newcommand{\C}{\mathbb{C}}

\newtheorem{thm}{Theorem}
\newtheorem{defn}[thm]{Definition}
\newtheorem{lemma}[thm]{Lemma}

\newtheorem{prop}[thm]{Proposition}
\newtheorem{cor}[thm]{Corollary}

\author{Brett Chenoweth}

\date{\today}

\begin{document}

\title{Carleman approximation of maps into Oka manifolds}
\begin{abstract} 
In this paper we obtain a Carleman approximation theorem for maps from Stein manifolds to Oka manifolds.
More precisely, we show that under suitable complex analytic conditions on a totally real set  \( M \) of 
a Stein manifold $X$, every smooth map \( X \rightarrow Y \)  to an Oka manifold $Y$ 
satisfying the Cauchy-Riemann equations along \( M \) up to  order \( k \) can be \( \mathscr{C}^k \)-Carleman approximated 
by holomorphic maps \( X \rightarrow Y \). Moreover, if \( K \) is a compact \( \mathscr{O}(X) \)-convex set such that \( K \cup M \) 
is  \( \mathscr{O}(X) \)-convex, then we can \( \mathscr{C}^k \)-Carleman approximate maps  which satisfy 
the Cauchy-Riemann equations up to order \( k \) along \( M \)  and are holomorphic on a neighbourhood of \( K \),
or merely in the interior of $K$ if the latter set is the closure of a strongly pseudoconvex domain. 
\end{abstract}

\date{\today}

\maketitle

%
%
\section{Introduction}\label{sec:intro}

In 1927 Carleman \cite{carleman1927} proved that for every pair of continuous functions 
 \( f, \epsilon \in \mathscr{C}(\mathbb{R}) \) such that \( \epsilon \) is strictly positive, there is an entire function \( g \) 
 on the complex plane $\C$ such that \( |f(x)-g(x)|< \epsilon(x) \) for every \( x \in \mathbb{R} \).
This improves Weierstrass's theorem on uniform approximation of continuous functions on compact intervals in $\R$
by holomorphic polynomials \cite{Weierstrass1985}. 

There are a number of generalisations of this theorem in the literature already, see Alexander \cite{alexander1979}, 
Gauthier and Zeron \cite{gauthier2002},  Hoischen \cite{hoischen1973}, Scheinberg \cite{scheinberg1976}, Manne \cite{manne1993} 
and Manne, \O{}vrelid and Wold \cite{manne2011}. These results concern Carleman approximation of \emph{functions}.
It is a  natural question to consider whether this type of approximation can be done for maps into Oka manifolds
as these are precisely the manifolds for which there exists an analogue of the Runge approximation theorem
for maps from Stein manifolds. 
For a comprehensive introduction to Oka theory, see Forstneri\v{c}'s book  \cite{forstneric2017} 
or Forstneri\v{c} and L\'{a}russon's survey article \cite{forstneric2011}.

Let \( \mathbb{N} = \{ 1, 2, \dots \} \). In this paper we prove the following  Carleman approximation result for maps from Stein manifolds to Oka manifolds.

%
%
\begin{thm} \label{main} 
Let \( X \) be a Stein manifold and \( Y \) be an Oka manifold. If \( K \subset X \) is a compact \( \mathscr{O}(X) \)-convex subset 
and \( M \subset X \) is a closed totally real set of class \( \mathscr{C}^r \) (\(r \in \mathbb{N}\)) which is \( \mathscr{O}(X) \)-convex, 
has bounded exhaustion hulls, and such that \( S = K \cup M \) is \( \mathscr{O}(X) \)-convex, then  for any \( k \in \{ 0,1, \dots, r \}\), the set \( S \) admits 
\( \mathscr{C}^k\)-Carleman approximation of maps \( f \in \mathscr{C}^k(X,Y) \) which are \( \bar \partial \)-flat to order \( k \) along \( M \) 
and holomorphic on a neighbourhood of \( K \).
\end{thm}

Intuitively, Theorem \ref{main} states the following. Given a map \( f \) satisfying the conditions mentioned above, there is \( g\in \mathscr{O}(X,Y) \) whose value and partial derivatives up to order \( k \) along $K \cup M$ are close to the value and  respective derivatives
of $f$, where closeness is measured with respect to fixed
systems of coordinate charts on the source and the target manifold.

At the end of \S \ref{oka} we explain how interpolation can be added to this result.

Observe that the conditions of Theorem \ref{main} are optimal due to Manne, \O{}vrelid and Wold's characterisation of totally real sets that admit \( \mathscr{C}^k \)-Carleman approximation \cite{manne2011}. 
By modifying our proof slightly, in a way that  will be made precise in \S \ref{oka}, we also obtain the following result.

\begin{thm}\label{thm2} Let \( X \), \( Y \), \( M \) and \( K \) be as in Theorem \ref{main}.
If in addition \( K \) is the closure of a strongly pseudoconvex domain, then for any \( k \in \{0,1, \dots, r \} \), the set \( S=K\cup M \) admits \(\mathscr{C}^k\)-Carleman 
approximation of maps \( f \in \mathscr{C}^k(X,Y) \) which are \( \bar \partial \)-flat to order \( k\) along \(S \). 
\end{thm}

The notions of holomorphic convexity for noncompact closed sets and of bounded exhaustion hulls are recalled in Subsection \ref{holconvexity}.
The notion of \( \mathscr{C}^k \)-Carleman approximation 
is recalled in Subsection \ref{ck}; see in particular Definition \ref{carleman}.

In Subsection \ref{flat} we define the class of maps which are \( \bar \partial \)-flat to order \( k \) along a closed set \( M \) and we introduce those sets \( M \) which are most important for the purposes of this paper, namely those which are totally real in the sense of Definition \ref{totallyreal}.
Intuitively, the former condition just means that these maps satisfy the Cauchy-Riemann equations to a certain order along \( M \). 
The condition in the theorem that \( f \in \mathscr{C}^k(X,Y) \) is \( \bar \partial \)-flat to order \( k \) along \(S \) implies in particular
that its restriction to the interior of $K$ is holomorphic. Clearly, this condition is necessary for \( \mathscr{C}^k \)-Carleman approximation
by holomorphic maps $X\to Y$ on $K\cup M$.

The proof of Theorem \ref{main}, given in Section \ref{oka},  involves an induction and reduction procedure making use of Poletsky's theorem \cite[Theorem 3.1]{poletsky2013} as well as methods 
from \cite{forstneric2017} and \cite{manne2011}.  Section \ref{c} is dedicated to the case where the target is \( \mathbb{C} \);
there we collect some results that will be needed in the proof of the general case of Theorem \ref{main}.

We finish this introduction with several remarks about the results proven in this paper.

Our first remark is that Theorem \ref{main} may be stated function theoretically as follows.
Let \( \mathscr{C}_S^k = \mathscr{C}_S^k(X,Y) \) denote the space of \( k \) times continuously differentiable maps from \( X \) to \( Y \) endowed with the Whitney topology (as defined in \cite{hirsch1976}).
Given a compact set \( K \) and a totally real set \( M \), let \( \mathcal{H}^k(X,Y; M,K) \) denote the subspace of all maps \( f \in \mathscr{C}_S^k(X,Y) \) which are \( \bar \partial \)-flat to order \( k \) along \( M \) 
and holomorphic on a neighbourhood of \( K \). 
For convenience we write \(  \mathcal{H}^k(X,Y; M,\emptyset)=  \mathcal{H}^k(X,Y; M)\).
Define the equivalence relation \( \sim \) on \( \mathscr{C}^k_S \) by \( f \sim g \) if and only if \( J_x^k f = J_x^k g \) for every \( x \in K \cup M \), where \( J^k_xf \) is the real \( k \)-jet of \( f \) at \( x \). 
The conclusion of  Theorem \ref{main} may then be stated as: 
\begin{align*}\mathscr{O}(X,Y)  \text{ is dense in }   \mathcal{H}^k(X,Y;M,K)/\sim.
\end{align*}

Secondly, in this paper we have considered  approximation of globally defined \( \mathscr{C}^k \)-maps which are  \( \bar \partial \)-flat along \( K \cup M \).
However, if \( M \) is a \( \mathscr{C}^r \)-submanifold rather than just a set,  we can just as easily approximate maps \( f \in \mathscr{C}^k(M, Y) \)  in the Carleman sense without any flatness assumptions on \( f \).
For using the tubular neighourhood theorem we may exend \( f \in \mathscr{C}^k(M,Y) \) to a tubular neighbourhood \( U \) of \( M \). 
We can then \( \bar \partial \)-flatten \( f \) near each \( p \in M \) without changing the values of \( f \) on \( M \) using \cite[Proposition 2.7]{manne2011} and a \( \bar \partial \)-flat partition of unity.
Hence we may assume that 
\begin{align} \label{fonnhd} f \in  \mathcal{H}^k(U,Y; M).
\end{align}
The submanifold \( M \) has a Stein neighbourhood basis by \cite[Corollary 3.2]{manne2011}.
So we can assume that \( U \) is Stein and hence our theorem applies. 
So \( \mathscr{O}(U,Y)|_{M} \) is dense in \( \mathscr{C}^k_S(M,Y) \). 
If in addition, \( f \) has a continuous extension, then by the Whitney approximation theorem we may take \( U = X\) in \eqref{fonnhd}. Hence we have
\begin{align*} \overline{\mathscr{O}(X,Y)|_{M}} = \{ f \in \mathscr{C}_S^k(M,Y): f  \text{ is continuously extendable to all of } X \},
\end{align*}
where the closure is taken inside of \( \mathscr{C}_S^k(M,Y) \).

Our third remark concerns the class of complex manifolds \( Y \) for which our theorem applies.
Clearly  a complex manifold \( Y \)  admitting Carleman approximation of maps in \(  \mathcal{H}^k(X,Y; M, K) \) whenever \( X, M, K \) satisfy the hypotheses of Theorem \ref{main} is necessarily Oka (because  the approximation property is trivially satisfied).
It is natural to ask whether the same is true when \( K = \emptyset \), in other words:\vspace{0.1cm}

\noindent \emph{Is a complex manifold \( Y \) that admits \( \mathscr{C}^k \)-Carleman approximation of maps in \(  \mathcal{H}^k(X,Y;M) \) for all \( X, M \) satisfying the hypotheses of Theorem \ref{main} necessarily Oka?}

Although we do not have a complete answer to this question, we can show as follows that such complex manifolds satisfy the weaker property of being dominable.
Suppose \( Y \) is a complex manifold admitting Carleman approximation of maps in \(  \mathcal{H}^k(X,Y; M) \) whenever \( X, M \) satisfy the hypotheses of Theorem \ref{main}.
Let \( m = \dim Y\) and  \( B_{\mathbb{R}} \) be the real unit ball in \( \mathbb{R}^{2m}\).
Suppose that \( p \in Y \) is an arbitrary point in \( Y \) and  \( \phi: U \rightarrow V \simeq B_{\mathbb{R}}  \) is a (complex) coordinate chart such that \( p \in U \), \( \phi(p) = 0 \) and, under the usual identification of \( \mathbb{C}^m \) with \( \mathbb{R}^{2m} \), \( \phi \) maps \( U \) onto \( B_{\mathbb{R}}\). 
Without changing the values near \( 0 \) we may extend the inverse map \( f \) smoothly to all of \( \mathbb{C}^m\).
Then we may apply Theorem \ref{main} to approximate \( f \) along \( M = \frac{1}{2}B_{\mathbb{R}} \) well enough so that the resulting holomorphic map \( \mathbb{C}^m \rightarrow Y \)  is a submersion at \( 0 \), as required.

Finally, we would like to conclude this introduction by indicating a few simple examples in which 
Theorem \ref{main} may be applied.
\begin{enumerate}[(a)]
\item \( K\subset \mathbb{C} \) is a closed disc centred at \( 0 \) and \( M = \mathbb{R} \subset \mathbb{C} \).
\item  \( K = \emptyset \) and \( M \) is a smooth unbounded curve in \( \mathbb{C}^N \).

\item \( K = \emptyset \) and \( M \) is defined as follows.
Write \( \mathbb{C}^N  = (\mathbb{R}^l \times \mathbb{R}^{N-l}) \oplus \, i \mathbb{R}^N\), \( 1 \leq l \leq N \). 
Let  \( \phi: \mathbb{R}^l \rightarrow \mathbb{R}^{N-l} \) and \( \psi: \mathbb{R}^l \rightarrow \mathbb{R}^l \) 
be \( \mathscr{C}^k \) functions such that for some \( \beta < 1 \), \( \psi \) satisfies the Lipschitz condition 
\begin{align*}
	 \| \psi(x) - \psi(x') \| \leq \beta \| x - x' \| 
\end{align*}
for all \(x, x' \in \mathbb{R}^l\).
Define
\begin{align*} M = \{ z =(x,y) + i w \in \mathbb{C}^N : y = \phi(x), \ w= \psi(x),\  x \in \mathbb{R}^l \}.
\end{align*}
\end{enumerate}

For Carleman approximation of functions, example (b) is due to Stolzenberg \cite{stolzenberg1966}, and Example (c) is due to 
Manne, \O{}vrelid and Wold \cite[Proposition 4.2]{manne2011}.

%
%
\section{Preliminaries}\label{sec:prelim}

Suppose \( U, V \) are open subsets of a complex manifold \( X \).
We say that \( U \) is relatively compact in \( V \) and write \( U \subset \subset V \)  if the closure 
\( \overline U \) is compact and \( \overline U \subset V \).

We adopt the standard convention that a function (or a map) $f$ is holomorphic on a compact set $K$ in a complex manifold $X$
if it is holomorphic in an unspecified open neighbourhood of $K$; the space of all such functions is denoted $\mathscr O(K)$.

\subsection{Holomorphic convexity and bounded exhaustion hulls}  \label{bounded} \label{holconvexity}Let \( X \) be a complex manifold.
For a compact set \( K \subset X\) the \( \mathscr{O}(X) \)-hull of \( K \) is defined by 
\begin{align*}  
	\widehat K_{\mathscr{O}(X)} = \{ z \in X : f(z) \leq \sup\limits_{K} |f| \text{ for every } f \in \mathscr{O}(X) \}.
\end{align*}
We will call a sequence of compact sets \( (M_j)_{j \in \mathbb{N}} \) in a topological space \( M \)  a \emph{normal exhaustion of \( M \)} if \( M = \bigcup_{j =1}^\infty  M_j \) and \( M_j \subset M_{j+1}^\circ \) for \( j \in \mathbb{N} \).

Let \( M \) be a closed, not necessarily compact subset of \( X \).
Following \cite{manne2011} we make the following definitions.
Define the \emph{hull} of \( M \) by
\begin{align*} \widehat{M}_{\mathscr{O}(X)}=\bigcup\limits_{j=1}^\infty\widehat{M_j}_{\mathscr{O}(X)},
\end{align*}
where \( (M_j)_{j \in \mathbb{N}} \) is a normal exhaustion of \( M \). 
The hull is independent of the choice of normal exhaustion.
For if \( (K_j)_{j \in \mathbb{N}} \) and \( (L_k)_{k\in \mathbb{N}} \) are  normal exhaustions of \( M \), 
then we can construct a new normal exhaustion by interweaving, that is 
\begin{align*} K_1 \subset L_{k_1}^\circ \subset L_{k_1} \subset K^\circ_{j_1} \subset \dots,
\end{align*}
where \( (j_l)_{l \in \mathbb{N}} \) and \( (k_{l})_{l \in \mathbb{N}} \) are increasing sequences of natural numbers.
The hull with respect to the new normal exhaustion is clearly the same as the hull with respect to \( (K_j)_{j \in \mathbb{N}} \) and the hull with respect to \( (L_k)_{k \in \mathbb{N}} \), as required.

We say that \( M \) is \( \mathscr{O}(X) \)-convex if \( \widehat{M}_{\mathscr{O}(X)}=M \).
For a compact set $M$ this coincides with the usual notion of $\mathscr{O}(X)$-convexity.

For a closed set \( M \subset X \), let \( h(M) \) denote the set
\begin{align*} h(M) = \overline{\widehat{M}_{\mathscr{O}(X)}\backslash M}.
\end{align*}
Thus, $M$ is \( \mathscr{O}(X) \)-convex if and only if $h(M)=\emptyset$.

Let \( (K_j)_{j \in \mathbb{N}}\) be a normal exhaustion of \( X \). We say that a closed set \( M \subset X \) has \emph{bounded exhaustion hulls} (or \emph{bounded E-hulls}) in \( X \) if the set \( h(K_j \cup M) \) is compact in \( X \) for every \( j \in \mathbb{N} \).
As before, we see that this property is independent of the choice of a normal exhaustion of $X$.

\begin{lemma} \label{lem:Lemma2}   
Let $X$ be a Stein manifold and $M\subset X$ be a closed $\mathscr{O}(X)$-convex subset.
If $K\subset X$ is a compact set contained in the interior of an $\mathscr{O}(X)$-convex set $L\subset X$
such that $h(K\cup M)\subset L$, then setting $K'=K\cup (M\cap L)$ we have that
\begin{equation}\label{eq:hullKM}
	\widehat{K\cup M} = \widehat {K'}\cup M.
\end{equation}
In particular, if $M$ has bounded exhaustion hulls then for every compact set $K\subset X$
the hull $\widehat{K\cup M}$ is closed and $\mathscr{O}(X)$-convex. 
\end{lemma}

\begin{proof}
Choose a compact $\mathscr{O}(X)$-convex set $M'\subset M$ such that $L\cap M\subset M'$. It suffices to 
show that for every such $M'$ we have that 
\begin{equation}\label{eq:hullKMprime}
	\widehat {K\cup M'}=\widehat {K'}\cup M'.
\end{equation}
Since $M$ is $\mathscr{O}(X)$-convex, this clearly this implies \eqref{eq:hullKM} 
and also shows that the set $\widehat{K\cup M}$ is closed and $\mathscr{O}(X)$-convex.

If \eqref{eq:hullKMprime} fails for some such $M'$, then the set $V:=\widehat {K\cup M'}\setminus \widehat {K'}\cup M'$ is nonempty.
Since $M'$ is $\mathscr{O}(X)$-convex, we have 
that $h(K\cup M')\subset h(K\cup M)\subset L$, where the second inclusion holds 
by the hypotheses of the lemma. This implies $V\subset L\setminus \widehat {K'}$.
Pick a point $p\in V$. There exists a holomorphic function $f\in\mathscr{O}(L)$ such that $|f(p)|=1$ and $|f|<1/2$
on $\widehat {K'}\subset L$. Since $L$ is $\mathscr{O}(X)$-convex, we may approximate $f$ uniformly on $L$ 
by a holomorphic function on $X$ with the same properties which we still denote $f$. 
Note that $V$ is a relative neighbourhood of $p$ in $\widehat {K\cup M'}$ whose 
relative boundary $b_r V := bV\cap \widehat {K\cup M'}$ is contained in $\widehat {K'}$. Since $f(p)=1$ and $|f|<1/2$
on $\widehat {K'}\supset b_rV$, we have a contradiction to Rossi's local maximum modulus principle \cite{Rossi1960}.
(A simple proof of Rossi's theorem was given by Rosay  \cite{rosay2006}.)
\end{proof}

\begin{lemma}\label{lem:Lemma3} 
If $X$ is a Stein manifold and $M\subset X$ is a closed $\mathscr{O}(X)$-convex subset with bounded 
exhaustion hulls, then there is a normal exhaustion $(K_j)_{j\in\mathbb N}$ of $X$ by compact $\mathscr{O}(X)$-convex 
sets such that $K_j\cup M$ is $\mathscr{O}(X)$-convex  for every $j\in\mathbb N$.
\end{lemma}

\begin{proof}
Choose any normal exhaustion $(K_j)_{j \in \mathbb{N}}$ of $X$. We proceed inductively. In the first step, choose a compact 
$\mathscr{O}(X)$-convex set $L_1$ such that $h(K_1\cup M) \subset L_1$, and let 
$K'_1$ denote the $\mathscr{O}(X)$-convex hull of $K_1\cup (L_1\cap M)$.
By Lemma \ref{lem:Lemma2} the set $K'_1\cup M$ is then closed and $\mathscr{O}(X)$-convex.
Next, choose $j_2>1$ such that $K'_1\subset K_{j_2}$. Pick 
a compact $\mathscr{O}(X)$-convex set $L_2\supset K_{j_2}$ such that 
$h(K_{j_2}\cup M) \subset L_2$, and let $K'_2$ denote the $\mathscr{O}(X)$-convex hull of 
$K_{j_2}\cup (L_2\cap M)$. As before we have that $K'_2 \cup M$ is closed and $\mathscr{O}(X)$-convex.
Clearly this process continues inductively and gives an exhaustion $(K'_j)_{j \in \mathbb{N}}$ of $X$ 
with the desired properties.
\end{proof}

%
%
\subsection{\(\mathscr{C}^k\)-Carleman approximation} \label{carlemanapprox} \label{ck}
Let \( X, Y \) be complex manifolds of complex dimensions \( m, n \) respectively and let \( A \) be a closed subset of \( X \).
Let \( f \) be a \(k \)-times continuously differentiable \( Y \)-valued map on a neighbourhood of \( A \). 
A \emph{Carleman pair for \( f \) with respect to \( A\)} is a pair \( P=(\mathscr{A}, \mathscr{B}) \), where \( \mathscr{A} = \{ (U_j,\phi_j) : j \in J \} \)   
and \( \mathscr{B} = \{ (V_j,\psi_j): j \in J \} \)  are collections of complex charts  on \( X \) and  \( Y \) respectively such that \( (U_j)_{j \in J} \) 
is a locally finite cover of \( U := \bigcup\limits_{j \in J} U_j \supset A \)  and for each \( j \in J \) the  following statements hold.
\begin{enumerate}[(i)]
\item There is a complex chart \( (\tilde U_j, \tilde \phi_j) \) on \( X \) such that \( U_j \subset \subset  \tilde U_j \) and \( \tilde \phi_j|_{U_j}=\phi_j \).
\item  \label{handy}There is a complex chart \( (\tilde V_j, \tilde \psi_j) \) on \( Y \)  such that \( V_j \subset \subset  \tilde V_j \) 
and \( \tilde \psi_j|_{U_j}=\psi_j \).
\item \label{relcomp} The image \( f(U_j \cap A) \) is relatively compact in \( V_j \).\end{enumerate}

The reason for the name is that these collections of charts help us to define Carleman approximation for general targets. 
To simplify notation we will often write a Carleman pair as \( P=((\phi_j, \psi_j))_{j \in J}\). Property (2) will  help show 
convergence in Theorem \ref{main}. 

Suppose \( P = ((\phi_j,\psi_j))_{j \in J}\) is a Carleman pair for some map \( f \) with respect to \( A \). 
If \( g, h \) are continuous \( Y \)-valued maps  on a neighbourhood of \( A \) such that 
\begin{align}\label{close1} g(U_j \cap A) \cup h(U_j \cap A) \subset\subset V_j
\end{align} 
for each \( j \in J \), then we say that \( g \) is \emph{\( P \)-close} to  \( h \) on \( A \). 
Hence both \( g \) and \( h \) map a small neighbourhood \( U'_j \) of \( U_j \cap A \) into  \( V_j \) for each \( j \in J \).

Let \( j \in J \) and \( g,h  \in \mathscr{C}^k(N,Y) \) be maps from an open neighbourhood 
\(  N \) of \( \overline{U_j} \cap A \)  to \( Y \) satisfying (\ref{close1}). We define
\begin{align*}
	e_j^{A,P}(g,h):= \begin{cases} \max\limits_{\substack{ |I| \leq k}} 
	\left\|  \dfrac{\partial^{|I|}(\psi_j \circ g)}{\partial \phi_j^I} - \dfrac{\partial^{|I|}(\psi_j \circ h)}{\partial \phi_j^I}\right\|_{A \cap U_j} 
	& \text{ if } U_j \cap A \not = \emptyset,  \\ 0 & \text{ if } U_j \cap A = \emptyset,
\end{cases}
\end{align*}
where \( \phi_j \) is treated as a real chart from \( U_j \) to \( \mathbb{R}^{2m} \) and \( | I| := I_1 + \dots +I_{2m}\)  for all multiindices \( I \in  \mathbb{N}^{2m} \).

\begin{defn}[\(\mathscr{C}^k\)-Carleman approximation] \label{carlemanapproximation} \label{carleman}
Let \( A \) be a closed subset of \( X \). 
We say that \( A \) admits {\( \mathscr{C}^k \)-Carleman approximation  of maps in 
a family \( \mathcal{F}\subset \mathscr{C}^k(X,Y) \)}  if for each \( f \in \mathcal{F} \) the following statement holds. 
For every Carleman pair \(P= ((\phi_j,\psi_j ))_{j \in J}\) of \( f \) with respect to \( A \) and  every family of strictly positive numbers 
\( (\epsilon_j)_{j \in J} \), there exists an entire map \( g \in \mathscr{O}(X,Y) \) such that \( g \) is \( P \)-close to \( f \) on \( A \) 
and \(e^{A,P}_j(f,g) < \epsilon_j \) for \( j \in J \).
\end{defn}

This definition is reminiscent of how the fine Whitney topology is defined on the mapping space
\( \mathscr{C}^k(M,N) \) (see \cite{hirsch1976}).

In the case where \( Y = \mathbb{C}^n \) we use Carleman pairs of the form \( ((\phi_j, \text{id}_{\mathbb{C}^n}))_{j \in J} \).
Whilst these are strictly speaking not Carleman pairs as they do not satisfy condition \eqref{handy}, the distinction is unimportant.
For once \( f \) and \( (\epsilon_j)_{j \in J} \) are specified, there are \( V_j \subset\subset \mathbb{C}^n \) such that showing 
Carleman approximation with respect to \( ((\phi_j, \text{id}_{\mathbb{C}^n}))_{j \in J} \) is the same as showing Carleman 
approximation with respect to \( ((\phi_j, \text{id}_{V_j}))_{j \in J} \) and the latter is an honest Carleman pair.

It is an elementary, albeit tedious, exercise to prove that the above definition of Carleman approximation is equivalent 
to the same statement with `every Carleman pair of \( f \)' replaced by `some Carleman pair of \( f \)'. The idea of proof is as follows.
Suppose we have proven that the statement of Definition \ref{carlemanapproximation} holds for some 
Carleman pair \( P = ((\phi_j, \psi_j))_{j \in J} \) of \( f \). 
Let \( P' = ((\phi_j', \psi_j'))_{j \in J'} \) be another Carleman pair for \( f \).
By the chain rule, derivatives of order up to \( k \) in one set of coordinates can be expressed as polynomials in derivatives 
of order up to \( k \) in another set of coordinates, with coefficients involving derivatives of the components of the transition map.
(Explicit formulas can be found in the literature for the multivariate chain rule for higher derivatives, for example see 
Constantine and Savits' paper \cite{constantine1996}.) Each \( A \cap U_j' \) is covered with finitely many \( U_k \).
Therefore, since we can approximate as well as we like on each \( U_k \) we can approximate as well as we like on \( U_j' \).

%
%
\subsection{\( \bar \partial \)-flat maps} \label{flat}
In this section we introduce the natural class of maps  that one wants to consider when studying
Carleman approximation by holomorphic maps.

Let \( X \) be a complex manifold of dimension \( m \) and let \( M \) be a closed subset of $X$.
Assume that $k\in\mathbb N$ and $f$ is a function of class \( \mathscr{C}^k \) in an open neighbourhood $U\subset X$ of $M$.
We say that $f$ is {\em \( \bar \partial  \)-flat to order \( k \) along \( M \)} if for each point \( x \in M \) there exist holomorphic 
coordinates \( z = (z_1, \dots, z_m) \) on $X$ around \( x \) such that in these coordinates 
\begin{equation}\label{eq:dibarflat}
	  D^\alpha (\bar \partial f)(x)= 0
\end{equation}
holds for all multiindices \( \alpha \in (\mathbb{N} \cup \{ 0 \})^{2m} \) with \( |\alpha| := \alpha_1 + \dots + \alpha_n \leq k -1 \).
Here, $D^\alpha$ denote the partial derivative with respect to the underlying real coordinates 
$x_1,y_1,\ldots, x_m,y_m$ with $z_j=x_j+i y_j$. If this holds, we shall write \( f \in \mathcal{H}_k{(X,\mathbb{C};M)} \),
where $\C$ denotes the target space (that is, we are considering functions).
It follows from chain rule that the definition is independent of the choice of local holomorphic coordinates.

If the set $M$ is sufficiently regular (for example, if it is a submanifold of $X$ of class $\mathscr C^1$) 
then the condition \eqref{eq:dibarflat} for all $|\alpha|\le k-1$ and $x\in M$ is equivalent to 
\[
	|\bar \partial f(z)| \le o(\mathrm{dist}(z,M)^{k-1}), 
\]
where $\mathrm{dist}$ is any Riemannian distance function on $X$ and the estimate is 
uniform on any compact subset of $M$.

More generally, let \( X, Y \) be complex manifolds of dimension \( m, n \) respectively and let \( M \) be a closed subset of 
\( X \). We say that a map \( f \in \mathscr{C}^k(X,Y) \) is \( \bar \partial  \)-flat to order \( k \) 
along \( M \) and we write \( f \in \mathcal{H}_k{(X,Y;M)} \) if for each point \( x \in M \) there exist an open neighbourhood 
$U\subset X$ of $x$ and holomorphic coordinates \( \zeta = (\zeta_1, \dots, \zeta_n) \) on an open neighbourhood $V\subset Y$ 
of \( f(x) \) such that $f(U)\subset V$ and each component function of the map $\zeta\circ f|U:U\to \mathbb C^n$ 
is \( \bar \partial  \)-flat to order \( k \) on $M\cap U$.
We declare that \( f \in \mathcal{H}_0(X,Y;M) \) for any continuous map \( f \in \mathscr{C}(X,Y) \).

Note that for a map \( f \) to admit \( \mathscr{C}^k \)-Carleman approximation  along a closed set \( M \), \( k \geq 1 \), 
 \( f \) must be \( \bar \partial  \)-flat up to order \( k  \) along \( M \).

The above definitions pertain to any closed subset $M$ of $X$. However, the most important case for the
purposes of this paper is when $M$ is a totally real submanifold of $X$ or, more generally, a totally real subset.

\begin{defn} \label{totallyreal}
A real submanifold \( M \) 
in a complex manifold \( X \) is called  {\em totally real} if 
the tangent space $T_x M$ at every point $x\in M$ (which is a real subspace of the complex 
tangent space $T_x X$) does not contain any nontrivial complex subspaces. 
A subset \( M \subset X \) is called a \emph{totally real set of class \( \mathscr{C}^k \)}, \( k \geq 1 \), if \( M \) is closed and locally 
contained in a totally real submanifold of class \( \mathscr{C}^k \),
that is, for every point \( x \in M \) there exist an open neighbourhood \( U \) of \( x \) and a 
closed \( \mathscr{C}^k \)  totally real submanifold \( M' \subset U \) such that \( M \cap U \subset M' \).
\end{defn}
These notions also apply to $\mathscr C^\infty$ and real analytic submanifolds.

It is well known that every function $f\in \mathscr C^k(M)$  on a  $\mathscr{C}^k$ totally real submanifold $M$ in a complex manifold $X$ extends to a function in $\mathscr{C}^k(X)$ which is smooth of class $\mathscr C^\infty$ on $X\backslash M$ and 
\( \bar \partial  \)-flat up to order \( k  \) along \( M \) (see for example  Hormander and Wermer \cite[Lemma 4.3]{hormander1968}).

Suppose \( K \) is a compact subset of \( X \). Let \(  \mathcal{H}_k(X,Y;M,K) \) denote the set of all  \(  f \in \mathcal{H}_k(X,Y;M) \) 
such that \( f| U \in \mathscr{O}(U,Y) \) for some neighbourhood \( U \) of \( K \). Note that 
\begin{align*} 
	\mathcal{H}_k(X,Y;M,\emptyset)=\mathcal{H}_k(X,Y;M). 
\end{align*} 
This is our motivation for choosing this notation.

\section{Target $ \mathbb{C} $}
\label{c}

In this section we specialise to the target \( \mathbb{C} \) case.
Here we collect several results that are needed in the proof of our main theorem.

When the set on which we are approximating on is compact, Carleman approximation coincides with Mergelyan approximation. 
Thus Propositions \ref{mergelyan}  is really  a result about Mergelyan approximation.
However, we introduce this theorem as follows because it is  easily stated with the terminology we have  introduced.

\begin{prop}\label{supportaway} \label{local} \label{mergelyan} \label{compactcase}
Let \( X \) be a Stein manifold, \( K \subset X \) be a compact \( \mathscr{O}(X) \)-convex set, and \( M \subset X \) be a compact, 
totally real set of class \( \mathscr{C}^k \) such that \( S:=K \cup M \) is \( \mathscr{O}(X)\)-convex. 
Given  \( f \in\mathcal{H}_k(X,\mathbb{C};K,M)\), a Carleman pair \( P = ((\phi_j,\psi_j))_{j \in \mathbb{N}}\) for \( f\) with respect to \(  K \cup M \) 
and  a sequence of positive real numbers \( (\epsilon_j)_{j \in \mathbb{N}} \), there is a function
\(
  g \in \mathscr{O}(X) 
\) such that  \( e_{j}(f,g) < \epsilon_j \) for \( j \in \mathbb{N} \). \vspace{0.1cm} \\
Moreover, given a closed analytic subvariety \( A \subset K \cup (X\backslash M) \) and a finite subset \(B \subset M \backslash K \)  such that \( f \) is holomorphic on (a neighbourhood of) \(  A\), 
we may choose \( g \) such that:
\begin{enumerate}[(i)]
\item \( f \) agrees with \( g \) (up to  some  finite order \( d \) of our choice) along \( A \),\label{cond1}
 \item \( f \) agrees with \( g \) up to order \( k \) at each \( p \in B \). \label{cond2}
\end{enumerate}
\end{prop}

\begin{proof} 

We first show that there exist  an approximating map \( g \in \mathscr{O}(X) \) satisfying \( \eqref{cond2} \). For \( \text{{Supp}}(f) \cap K = \emptyset \) the proof  is the same as that of Manne, \O{v}relid and Wold \cite[Proposition 3.13]{manne2011} except that one may need to apply Oka-Weil theorem to a slightly larger \( \mathscr{O}(X) \)-convex set so as to allow for the derivatives of \( f \in \mathcal{H}_k(X,\mathbb{C};K,M) \) to be approximated on all of \( K \).

The case where \( \text{{Supp}}(f) \cap K \not = \emptyset \) is then easily reduced to the above case as follows.
Take an arbitrary  \( f\in H_k(X,\mathbb{C};M,K) \) with \( \text{{Supp}}(f) \cap K \not = \emptyset \). 
Suppose that \( P=((\phi_j,\text{id}_{\mathbb{C}}))_{j=1}^n \) is a Carleman pair for \( f \) and
 \( \epsilon_1, \dots, \epsilon_n > 0 \).

Let \( U \) be a neighbourhood of \( K \) such that \( f|_{U} \) is holomorphic.
Since \( K \) is \( \mathscr{O}(X) \)-convex and \( X \) is Stein, 
there are compact \( \mathscr{O}(X) \)-convex subsets \( K', K''  \) such that 
\begin{align*} K \subset (K')^\circ \subset K'\subset (K'')^\circ \subset K'' \subset U.
\end{align*}
We may assume that \( K'' \backslash K\) does not intersect \(B \).
Let \( V \subset \subset K'\) be an open neighbourhood of \( K \).
Let \( \chi  \in \mathscr{C}_0^\infty(X) \) be a smooth cut-off function such that \( \chi \equiv 1 \) on \( V \) and 
\( \text{Supp} \chi \subset (K')^\circ\).
We may assume that \( \chi \) is \( \bar \partial \)-flat up to order \( k \) along \( M \) by \cite[Proposition 2.7]{manne2011}.
Let \( C >1 \) be  such that \( e_j^{S, P}(\chi \theta,0) \leq C e_j^{S, P}(\theta, 0) \) for every 
\( \theta \in \mathscr{C}^k(X) \) \( P \)-close to \( f \) on \( S \).

Since the set $K''\subset U$ is \( \mathscr{O}(X) \)-convex and \( S \cap K' \subset (K'')^{\circ} \), 
the Oka-Weil theorem 
yields a function 
\( \tilde{f} \in \mathscr{O}(X) \) such that
\begin{align*}
 	e_j^{S \cap K', P}(f,\tilde f) < \dfrac{\epsilon_j }{2C}, \quad \text{ for } j =1, \dots, n.
\end{align*}
By the special case considered above, there exists  \( h \in \mathscr{O}(X) \) that agrees with \[ (1- \chi) \cdot (f - \tilde f) \] to order \( k \) at each \( p \in B \)  and such that 
\begin{align*} e_j^{S,P}(h,(1-\chi)\cdot(f-\tilde f))< {\epsilon_j}/{2C}
\end{align*}
for  \( j =1 , \dots, n \).  Let \( g = \tilde{f}+ h\). Observe that \( g \) agrees with \( f \) up to order \( k \) at each \( p \in B \). Moreover,
\begin{align*} e_j^{S, P}(g,f) &= e_j^{S, P}(\tilde f + h, f) \\
&= e_j^{S, P}(h, f - \tilde f)\\
&\leq e_j^{S, P}(h, (1 - \chi)(f- \tilde f)) + e_j^{S,P}((1-\chi)(f-\tilde f), f -\tilde f)\\
&= e_j^{S, P}(h, (1 - \chi)(f- \tilde f)) + e_j^{S,P}(\chi(f-\tilde f), 0)\\
&<\epsilon_j/2 + (\epsilon_j/2 C) \cdot C = \epsilon_j
\end{align*}
for \( j =1, \dots, n \), as required.

We may build (jet) interpolation along \( A \) into our construction precisely as in \cite[p. 89]{forstneric2017}.
That is, using Cartan's theorems  A and  B  we  write
\begin{align*} f = \phi + \sum_{\nu} h_\nu f_\nu \quad \text{ on } U
\end{align*}
where \( \phi \) is a holomorphic function on \( X \) that agrees with \( f \) up to order \( d \) along \( A \),
\( h_v \) is a finite collection of holomorphic functions on \(X \) that vanish up to order \( d \) along \( A \) and whose common zero set is \( A \) and \( f_\nu \) is a finite collection of functions in \(H_k(X,\mathbb{C};M,K)\).
Moreover, using  \( \bar \partial \)-flat partitions of unity we can choose  \( f_\nu \) so that \( f \) agrees with \(  \phi + \sum_{\nu} h_\nu  f_\nu \) up to order \( k \) along \( M \).
Then approximating each \( f_\nu \) with interpolation along \( B \) as above, we obtain
\begin{align*}\tilde  f = \phi + \sum_{\nu} h_\nu \tilde f_\nu,
\end{align*}
which clearly has the desired properties.
\end{proof}

%
%
\section{Oka target} \label{oka}

In this section we prove Theorem \ref{main}. Note that it suffices to consider the case where \( r =k \) by the Whitney approximation theorem.

The method of proof is as follows.
First we will prove Lemma \ref{induction} which is a Mergelyan type theorem with jet approximation and interpolation. 
Lemma \ref{induction}, or rather an easy consequence of it, Corollary \ref{inductionstep}, is then used to recursively construct a sequence of 
maps \( (f_l)_{l \in \mathbb{N}} \) whose limit approximates \( f \)  uniformly on \( K \) and  in the Carleman sense along \( M \), 
where \( f \in \mathcal{H}_k(X,Y; M, K) \) (see the proof of Theorem \ref{submain}).
Finally, given \( f \in \mathcal{H}_k(X,Y; M, K) \) we may apply Lemma \ref{induction} to \( f \) and then use 
Theorem \ref{submain} to get a map approximating \(f \) well in the Carleman sense along all of  \( K \cup M \),
as required. 

Hence Theorem \ref{main}  follows once we have proven  Lemma \ref{induction}, Corollary \ref{inductionstep} and Theorem \ref{submain}.

\noindent \emph{Remark.} If in addition \( K \) is the closure of a strongly pseudoconvex domain, then \( S \) admits \(\mathscr{C}^k\)-Carleman 
approximation of maps \( f \in \mathscr{C}^k(X,Y) \) which are \( \bar \partial \)-flat to order \( k \) along \(S \). 
For we may repeat the above arguments using \cite[Theorem 26]{fornaess2018}, or rather its proof, instead of Proposition \ref{compactcase} 
to find the map  \( h' \) in our proof of Lemma \ref{induction}.\vspace{0.1cm}

\begin{lemma} \label{fhat} \label{induction} Let \( X, Y, K, M \) be as in Theorem \ref{main}. Let \( L \supset L^\circ \supset K \) be 
an \( \mathscr{O}(X) \)-convex set. Given a  map \( f \in \mathcal{H}_k(X,Y;K,M)\), a Carleman pair \( P = ((\phi_j,\psi_j))_{j \in \mathbb{N}}\) 
for \( f\) with respect to \( S=K \cup M \) and  a sequence of positive real numbers \( (\epsilon_j)_{j \in \mathbb{N}} \), there is a map
\begin{align*}
 g \in \mathcal{H}_k(X,Y;L,M) 
\end{align*}
such that \( e_{j}^{S, P}(f,g) < \epsilon_j \) for \( j \in \mathbb{N} \). 
\end{lemma}

Due to the technical nature of the proof of Lemma \ref{induction} we provide the following sketch.
On the compact piece \( S \cap L'' \) of \( S \), the problem of approximating maps is (locally) reduced to the problem of approximating functions using the holomorphic map \( \Phi \) defined in Equation \ref{Phi}.
Then the local solution is patched together with \( f \) to give us a globally defined approximating map \( \hat f \) which is holomorphic in a neighbourhood of \( S \cap L' \) (see Equation \ref{fhatt}).
We then approximate \( \hat f \) by a homotopic map \( \tilde f \) which is holomorphic on \( L'' \) and approximates \( \hat f \) well on \( S \cap L' \). Using the homotopy from \( \hat f \) to \( \tilde f \), which we change slightly near \( (L' \backslash L) \cap M \)  if necessary, we patch \( \hat f \) to \( \tilde f\) outside a neighbourhood of \( L \) so that the gluing has the desired approximation property.

\begin{proof}

Let \( L' \), \(L '' \) be compact \( \mathscr{O}(X) \)-convex subsets of \( X \) such that
\begin{align*} K \subset L^\circ \subset L \subset (L')^\circ \subset L' \subset (L'')^\circ \subset L''.
\end{align*}

\begin{center}
\includegraphics[scale=0.60]{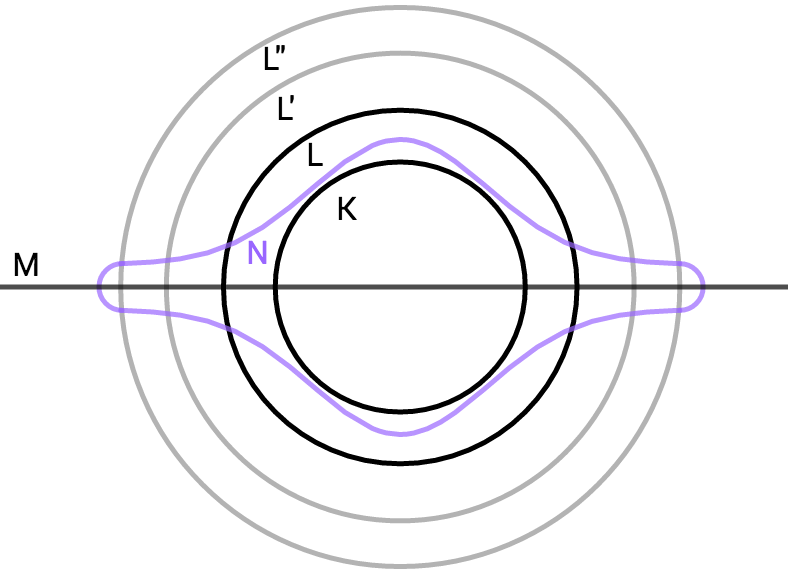}
\end{center}

Let \( F : X \rightarrow X \times Y \),
\begin{align*}
   F(x)= (x,f(x)).
\end{align*}
The set \( F(S \cap L'') \) has a Stein neighbourhood basis by Poletsky's theorem \cite[Theorem 3.1]{poletsky2013}.
Hence we may construct a fibre preserving holomorphic embedding 
\begin{align*}
 G : W \rightarrow X \times \mathbb{C}^n 
\end{align*} of a Stein neighbourhood \( W \supset F(S \cap L'') \) into \( X \times \mathbb{C}^n \)  (for some \( n \in \mathbb{N} ) \) 
as in \cite[Lemma 3.4.3]{forstneric2017}. Let 
\begin{align}
 \label{nk}H := G \circ F |_{N}  \in  \mathcal{H}_k(N, X \times \mathbb{C}^n; M,K),
\end{align}
where \( N \) is a Stein neighbourhood of \( S \cap L'' \) contained in \( F^{-1}(W) \). 
We may write \( H(x) = (x,h(x)) \) since \( G \) is fibre preserving.
Then  \( Q = ((\phi_j, \text{id}_{\mathbb{C}^n}))_{j \in \mathbb{N}} \)
 is a Carleman pair for \( h \) with respect to \( S \).
By \cite[Theorem 3.3.5]{forstneric2017},  there exist a Stein neighbourhood \( \Omega \) of \( G(W) \) and a fibre preserving holomorphic retraction \( r : \Omega \rightarrow G(W) \).
Let 
\begin{align} \label{Phi} \Phi = p_2 \circ G^{-1} \circ r : \Omega \rightarrow Y.
\end{align}
Observe that \begin{align} \label{fk}
 f = p_2 \circ G^{-1} \circ H= p_2 \circ G^{-1} \circ r \circ H=\Phi \circ H.
\end{align}
on \( N \).

Let \( \chi, \chi' : X \rightarrow [0,1] \) be smooth functions, \( \bar \partial \)-flat to order \( k \) along \( M \), such that \( \chi, \chi' \) equal \( 1 \) on neighbourhoods of \( L \), \( L' \) respectively,  \( \chi|_{X \backslash L'} = 0 \) and  \( \chi'|_{X\backslash L''} = 0 \).

Choose \( \delta >0 \) so small that:
\begin{enumerate}
\item The \( 4\delta \)-tube around the graph of \( h \) over \( S \cap L'' \) is a relatively compact subset of \( \Omega \), that is,  \label{contained} \begin{align*} T_{4\delta}=\{ (x,z) \in (S \cap L'') \times \mathbb{C}^n : \|h(x) - z \| < 4 \delta \} \subset \subset \Omega.
\end{align*}
\item \label{rhoclose} \label{intersection} If \(  \theta \in \mathscr{O}(S \cap L'',\mathbb{C}^n) \) satisfies \( e_j^{S \cap L'',Q}(h, \theta) < \delta\) for some \( j \in \mathbb{N} \), then
\begin{align*}
  e_j^{S \cap L'',P}(f,\Phi \circ ((1- \chi')H+ \chi'\Theta)) < \epsilon_j/2,
\end{align*}
where \( \Theta (x) = (x,\theta (x))  \). 
(Such \( \delta  \) exist by \eqref{fk} and the chain rule since \( \Phi \) is holomorphic and hence \( \mathscr{C}^{k+1}\).)
\end{enumerate}

Now choose  \( h' \in \mathscr{O}(N,\mathbb{C}^n)\)  such that \begin{align}\label{hkhkhat}\| h -  h' \|_{S \cap L''} \leq \max_je_j^{S \cap L'',Q}(h,  h') < \delta/4.
\end{align} 
and \( H'(N) \subset G(W) \), where \( H'(x) = (x, h'(x)) \).
Such a map can be found using Proposition \ref{mergelyan} and the retraction \( r \).

Let \( N' \subset \subset N \) be a Stein neighbourhood of \( S \cap L'' \) so small that 
\begin{align} \label{nkprime} T_{2{\delta}}=\{ (x,z) \in N'\times \mathbb{C}^n : \| h(x) - z \| <  2 \delta \} \subset \Omega
\end{align}
and 
\begin{align*} \| h -  h' \|_{N'} < \delta/2.
\end{align*}
We may modify \( \chi' \) outside of a small neighbourhood of \( S \cap L'' \) so that \( \chi' \) has support in \( N' \). 
In particular, \( \chi' \) is still \( 1 \) on an open neighbourhood \( N''\subset N' \) of \( S \cap L' \).
Define \( \hat f: X \rightarrow Y \), 
\begin{align} \label{fhatt} \hat f(x) = \begin{cases} p_2 \circ G^{-1} \circ r \circ (\chi' H' + ( 1 - \chi')H)(x) & \text{ if } x \in N, \\
f(x) & \text{ if } x  \in X \backslash N,
\end{cases}
\end{align}
Note that \( \hat f \) is holomorphic on \( N'' \). Moreover, \begin{align*} e_j^{S,P}(\hat f, f) < \epsilon_j /2
\end{align*}
for every \( j \in \mathbb{N} \) by our choice of \(  h' \). Let 
\begin{align*} \widehat H = \chi' H' + ( 1 - \chi')H,
\end{align*}
so that \( \widehat f = \Phi \circ \widehat H \) on \( N \). 

Let \( K' \) be a compact \( \mathscr{O}(X) \)-convex neighbourhood of \( S \cap L' \) contained in \( N''\cap  (L'')^\circ\). 
Such a neighbourhood exists by \cite[Theorem 5.1.6]{hormander1990} since the sublevel sets  of a strictly plurisubharmonic exhaustion function  are \( \mathscr{O}(X)\)-convex. 

Using the  basic  Oka property with approximation one obtains a sequence  of holomorphic  maps \( (\hat f_{\nu}: X \rightarrow Y)_{\nu \in \mathbb{N}} \)  converging uniformly to \( \hat f \) on \( K' \)  such that each  \( f_\nu \) is homotopic to \( \hat f \).

Clearly \( \widehat F_{\nu} \rightarrow \widehat F \) uniformly on \( K' \) as \( \nu \rightarrow \infty \), where 
\( \widehat F_{\nu}(x) = (x, \hat f_{\nu}(x)) \) and \( \widehat F(x) = (x, \hat f(x)).
\)
Since 
\begin{align}
 \widehat  H(N'') = H'(N'') \subset H'(N)  \subset G(W),
\end{align} it follows from the definition of \( \hat f \) that \( \widehat H = G \circ \widehat F \) on \( N'' \).
Hence
\begin{align*} p_2 \circ G \circ \widehat F_{\nu} \rightarrow p_2 \circ G \circ \widehat F= \hat h
\end{align*}
uniformly on \( K' \) as \( \nu \rightarrow \infty \) because \( p_2 \circ G \) is uniformly continous on a neighbourhood of \( \widehat F(K') \).
Since all sequences are sequences of holomorphic maps and uniform convergence of holomorphic maps implies uniform convergence of derivatives, it follows that we may find \( \nu_0 \in \mathbb{N} \) so large that the following conditions hold:
\begin{enumerate}[(\(a\))]
\item The graph of \( \hat f_{\nu_0} \) above \( S \cap L' \) is contained in \( G^{-1}(T_{2 \delta})\subset W \).
\item \label{tildehkhkhat}  \( \| \hat h - \tilde h \|_{S \cap L'} \leq \max_je_j^{S \cap L',Q}(\hat h, \tilde h) < \delta/2,
\) where \( \widetilde H = G \circ F_{k,\nu_0} \) and  \begin{align*} \tilde h = p_2 \circ \widetilde H. \end{align*}
\item For every \( j \in \mathbb{N} \), 
\begin{align*} e_{j}^{S \cap L',P}( \hat f , \Phi \circ (\chi \tilde H + ( 1 - \chi) \widehat H)) < \epsilon_j/2.
\end{align*}

\end{enumerate} 
Let  \( \tilde f = \hat f_{\nu_0} \). 
Note that
\(\tilde f = \Phi \circ \widetilde H.
\)

Let \( \hat f(x,t) \) be a homotopy from \( \hat f \) to \(  \tilde f \) such that 
\begin{align*} \hat f(\cdot,t) = \Phi \circ ((1-t) \widehat H+ t \widetilde H),
\end{align*}
near \( (L' \backslash L) \cap M \).
Such a homotopy can be obtained by gluing any given homotopy that approximates \( \hat f \) well enough to 
\begin{align*} \Phi \circ ((1-t) \widehat H+ t \widetilde H). \end{align*}
outside of a neighbourhood of \((L' \backslash L) \cap M \). (For example, using a smooth partition of unity at the level of functions.)

One can check that \( g: X \rightarrow Y \), 
\begin{align} \label{mapg} g(x) =\hat  f(x,\chi(x)),
\end{align}
is the desired approximating function.
\end{proof}

\begin{cor}\label{inductionstep}Let \( X, Y, K, M, L \) be as in Lemma \ref{fhat}. Given 
\( f \in \mathcal{H}_k(X,Y;K,M)\), 
a Carleman pair \( P = ((\phi_j,\psi_j))_{j \in \mathbb{N}}\) for \( f\) with respect to \( M \) and a sequence of positive real numbers \( (\epsilon_j)_{j \in \mathbb{N}} \), there is  a map
\begin{align*}
 \hat f \in \mathcal{H}_k(X,Y;L,M) 
\end{align*}
such that \( e_{j}(f,\hat f) < \epsilon_j \) for \( j \in \mathbb{N} \) and \( \hat f \) approximates \( f \) uniformly on \( K \).
\end{cor}
\begin{proof} Clearly by adding finitely many pairs \( P \) can be made into a Carleman pair for \( f \) with respect to \( K \cup M \). Then we apply Lemma \ref{fhat} for a suitably chosen sequence \( (\epsilon_j)_{j \in \mathbb{N}} \) to obtain the desired result.
\end{proof}

\begin{thm} \label{submain} Let \( X \), \( Y \), \( K \), \( M \) be as in Theorem \ref{main}. 
Given  \( f \in \mathcal{H}_k(X,Y;K,M)\), a Carleman pair \( P = ((\phi_j,\psi_j))_{j \in \mathbb{N}}\) for \( f\) with respect to \(  M \) 
and  a sequence of positive real numbers \( (\epsilon_j)_{j \in \mathbb{N}} \), there is a map
\begin{align*}
 \bar f \in \mathscr{O}(X,Y) 
\end{align*}
such that \( \bar f \) approximates \( f \) uniformly on \( K \) and  \( e_{j}(f,\bar f) < \epsilon_j \) for \( j \in \mathbb{N} \). 
\end{thm}

\begin{proof}

Let \( f \in \mathscr{C}^k(X,Y) \) be \( \bar \partial \)-flat to order \( r \) along \( M \) 
and holomorphic on a neighbourhood of \( K \) and let \( P  =((\phi_j,\psi_j))_{j \in \mathbb{N}} \) be a Carleman pair for \( f \) with respect to \(  M \).
Suppose  \( \epsilon_0, \epsilon_1, \epsilon_2, \dots \) is a sequence of positive real numbers. 
Without loss of generality we may suppose that for each \( j \in \mathbb{N} \),
\begin{align} \label{closecall} \epsilon_j < \inf\limits_{y \in f(U_j \cap S)}d_u(\psi_j(y) ,\mathbb{C}^N \backslash  \psi_j(V_j)),
\end{align}
where \( U_j \) is the domain of \( \phi_j \), \( V_j \) is the domain of \(\psi_j \) and \( d_u \) is the usual metric on \( \mathbb{C}^N \) (\( N = \dim Y \)). Observe that this distance is positive by \eqref{relcomp} of the definition of Carleman pair.
In what follows we will denote the domain of a map \( \phi \) by \( \mathcal{D}(\phi) \).

Let \( (K_l)_{l\in\mathbb{N}} \) be a normal exhaustion of \( X \) by compact \( \mathscr{O}(X)\)-convex sets such 
that the set \( S_l = K_l \cup M \) is \( \mathscr{O}(X) \)-convex for every \( l \in \mathbb{N} \) (see Lemma \ref{lem:Lemma3}).
We may assume that \( K_1 = K \). Let \( d \) be a distance function on \( Y \) induced by a Riemannian metric.

By Corollary \ref{inductionstep}, we can recursively construct a sequence \( (f_l)_{l \in \mathbb{N}} \) of maps  \( f_l \in  \mathcal{H}_k(X,Y;M,K_l)\)
(this notation was introduced in \S \ref{flat})  such that for each \( l \in \mathbb{N} \),
\( f_{l+1} \) is \( P \)-close to \( f_{l} \) on \( M \),
\begin{align}
 \label{uniformly}\sup\limits_{x \in K_l} d(f_{l+1}(x),f_l(x)) &< \epsilon_0/2^l, \text{ and }\\
 \label{inductionhyp}e_j^{S_{l},P}(f_{l+1},f_{l}) &< \epsilon_j/2^{l},
\end{align}
for \( j \in \mathbb{N} \).
This sequence clearly converges uniformly on compacts since \( (K_{l})_{l \in \mathbb{N}} \) is a normal exhaustion of \( X \).

Now let us show that \(\bar f \) is the desired approximating function.
Note  that \( \bar f \) is holomorphic since it is the uniform limit of holomorphic maps.
Moreover, one can show that \( \bar f\) is \( P \)-close to \( f \) using the inequality \eqref{closecall} and   \eqref{inductionhyp}. 
Note that \eqref{uniformly} easily gives us that \( \bar f \) approximates \( f \) uniformly on \( K \).
It remains to show that 
\begin{align*} e_j^{S_1,P}(f,\bar f) < \epsilon_j
\end{align*}
for every \( j \in \mathbb{N}. \) 
It follows from (i) and the fact that \( e_j^{S_1,P} \) satisfies the triangle inequality that for each \( j \in \mathbb{N} \) and each \( k \in \mathbb{N}\),
\begin{align} \label{est} e_{j}^{S_1,P}(f_1,f_k)
< e_j^{S_1,P}(f_1,f_2) + \epsilon_j/2.
\end{align}

For each \( j \in \mathbb{N} \) there is a chart  \( \tilde \phi_j \) which extends \( \phi_j \)  and for which 
\begin{align*} \mathcal{D}(\phi_j) \subset \subset \mathcal{D}(\tilde \phi_j)
\end{align*}
(this is property \eqref{handy} in the definition of Carleman pair).
Also, for \( l \) sufficiently large \( f_l \) is holomorphic on a neighbourhood of \( \overline{\mathcal{D}(\phi_j)} \) and since \( f_l \rightarrow \bar f \) uniformly, it follows that 
\begin{align*} (\psi_j \circ f_l \circ \tilde \phi_j)_{l \in \mathbb{N}}
\end{align*}
is a sequence of holomorphic functions on some  (fixed) neighbourhood \( U \)  of \( \overline{\mathcal{D}(\phi_j^1)}  \) converging uniformly to \( \psi_j \circ \bar f \circ \tilde \phi_j \)  on \( U \) as \( l \rightarrow \infty \).
Hence we have uniform convergence of all derivatives on \( \overline{\mathcal{D}(\phi_j^1)}  \) which is a compact subset of \( U \). 
Hence it follows  from \eqref{est} that for each \( j \in \mathbb{N} \),
\begin{align*} e_j^{S_1,P}(f,\bar f) = \lim\limits_{k \rightarrow \infty} e_j^{S_1,P}(f,f_k) \leq e_j^{S_1,P}(f_1,f_2) + \epsilon_j/2 < \epsilon_j,
\end{align*}
as required.
\end{proof}

This concludes our proof of Theorem \ref{main} which was outlined in the beginning of this section.

We finish this paper with a discussion of how interpolation can be added to Theorem \ref{main}. The two modes of interpolation that we consider are as follows:
\begin{enumerate}[(a)]\item \label{statementa} If \( A\subset K \cup (X\backslash M) \) is a closed analytic subvariety on which \( f \) is holomorphic, then the approximating map \( g \) can be chosen to agree with \( f \) along \( A \). \item  \label{statementb}  If  \( A \subset M \backslash K \), \( B \subset K \cup (X\backslash M)  \) are closed discrete subsets so that \( f \) is holomorphic on a neighbourhood of \( B \) and  \( (d_{p})_{p \in B} \) is a family of non-negative integers, then we may choose the approximating map \( g \) to agree with \( f \) up to order \( k \) at each \( p \in A \) and up to order \( d_p \) at each \( p \in B \).
\end{enumerate}

Both statements can be proven in essentially the same way as Theorem \ref{main}.
In each case the only appreciable difference is in induction step, that is, Lemma \ref{induction}.
We shall now explain precisely what these differences are.

For Statement (a) first one uses Proposition \ref{compactcase} to obtain \( h' \) which approximates \( h \) and interpolates along \(A \cap N \). 
Then the map \( \hat f \) defined as in \eqref{fhatt} agrees with \( f \) along \( A \).
Thus using the basic Oka property with approximation and interpolation (BOPAI) we may take \( \hat f(\cdot, t) \) (and hence \( g \)) to agree with \( f \) along \( A \), as required.

For Statement (b) one chooses \( \chi, \chi' \) to take values in \( \{0,1\} \) on a neighbourhood of \( A \cup B \).
Using Proposition \ref{compactcase} we may assume that \( h' \) jet interpolates \( h \) at points of \( A \cup B \)  in its domain \( N \). Then \( \hat f \) jet interpolates \( f \) along \( A \cup B \) because \( \chi' \) takes values in \( \{ 0,1\} \) in a neighbourhood of \( A \cup B \). 
Thus by the basic Oka property with jet interpolation we may choose \( \tilde f\) jet interpolating \(\hat  f \) (and hence \( f \)) along \( (A \cup B) \cap \chi^{-1}(1) \). Finally, because \( \chi\) takes values in \( \{0,1\} \) it follows that the map \( g \) defined in  \eqref{mapg} jet interpolates \( f \) along \( A \cup B \).

Jet interpolation along a closed analytic subvariety \( A \) of positive dimension seems more difficult  because we cannot separate the gluing set from the set on which we wish to interpolate  as we have done for Statement (b) above.
Hence there seems to be no reason to expect that the map  \( \hat f \) defined in \eqref{fhatt} is holomorphic on a neighbourhood of \( A \) and  no reason to expect that we could apply the basic Oka property with jet interpolation to \( \hat f \) as we would like.

\subsection*{Acknowledgements} The author is supported by  grant MR-39237 from ARRS, Republic of Slovenia, associated to the research
program P1-0291 {\em Analysis and Geometry}. 
The author would like to thank F. Forstneri\v{c} for his support and guidance.

\bibliographystyle{hacm}
\bibliography{PaperarXiv2}

\end{document}